\documentclass[a4paper,11pt]{article}
\usepackage[left=3.17cm,right=3.17cm,top=2.54cm,
headheight=0.5cm,headsep=0.54cm,bottom=2.54cm,footskip=0.79cm
]{geometry}

\usepackage[all]{xy}
\usepackage{amsmath,amsfonts,amssymb,amsthm,epsfig,amscd,comment,latexsym,psfrag}
\usepackage{nicefrac,xspace,tikz}

\usetikzlibrary{arrows}
\usetikzlibrary{decorations.markings}
\def\Z{\mathbb Z}

\def\C{\mathbb C}

\def\1{{\bf{1}}}

\usepackage[pdfstartview=FitH]{hyperref}

\def\onebox{\begin{tikzpicture}
\draw (0,0)rectangle(0.25,0.25);
\draw (0,0.25) -- (0.175,0.35) [-];
\draw [shift = {+(0.25,0)}](0,0.25) -- (0.175,0.35) [-];
\draw [shift = {+(0.25,-0.250)}](0,0.25) -- (0.175,0.35) [-];
\draw (0.175,0.35) -- (0.425,0.35) [-];
\draw (0.425,0.35) -- (0.425,0.1) [-];
\end{tikzpicture}}
\def\twoboxy{\begin{tikzpicture}
\draw (0,0)rectangle(0.5,0.25);
\draw (0.25,0) -- (0.25,0.25) [-];
\draw (0,0.25) -- (0.175,0.35) [-];
\draw [shift = {+(0.25,0)}](0,0.25) -- (0.175,0.35) [-];
\draw [shift = {+(0.5,0)}](0,0.25) -- (0.175,0.35) [-];
\draw [shift = {+(0.5,-0.250)}](0,0.25) -- (0.175,0.35) [-];
\draw (0.175,0.35) -- (0.675,0.35) [-];
\draw [shift = {+(0.25,0)}] (0.425,0.35) -- (0.425,0.1) [-];
\end{tikzpicture}}
\def\twoboxz{\begin{tikzpicture}
\draw (0,0)rectangle(0.25,0.5);
\draw (0,0.25) -- (0.25,0.25) [-];
\draw [shift = {+(0,0.25)}](0,0.25) -- (0.175,0.35) [-];
\draw [shift = {+(0.25,0.25)}](0,0.25) -- (0.175,0.35) [-];
\draw [shift = {+(0.25,0)}](0,0.25) -- (0.175,0.35) [-];
\draw [shift = {+(0.25,-0.250)}](0,0.25) -- (0.175,0.35) [-];
\draw (0.425,0.625) -- (0.425,0.1) [-];
\draw [shift = {+(0.175,0.35)}] (0,0.25) -- (0.25,0.25) [-];
\end{tikzpicture}}
\def\twoboxx{\begin{tikzpicture}
\draw (0,0)rectangle(0.25,0.25);

\draw (0,0.25) -- (0.35,0.45) [-];
\draw [shift = {+(0.25,0)}] (0,0.25) -- (0.35,0.45) [-];
\draw [shift = {+(0.25,-0.25)}] (0,0.25) -- (0.35,0.45) [-];
\draw [shift = {+(0.175,0.35)}](0,0) -- (0.25,0) [-];
\draw [shift = {+(0.25,-0.25)}][shift = {+(0.175,0.35)}](0,0) -- (0,0.25) [-];
\draw [shift = {+(0.35,0.45)}](0,0) -- (0.25,0) [-];
\draw [shift = {+(0.175,0.1)}] [shift = {+(0.25,-0.25)}][shift = {+(0.175,0.35)}](0,0) -- (0,0.25) [-];
\end{tikzpicture}}

\def\footnoterule{\kern 1mm \hrule width 7cm \kern 2.2mm}%

 \newtheorem{thm}{Theorem}[section]
 \newtheorem{prp}[thm]{Proposition}
 \newtheorem{lem}[thm]{Lemma}
 \newtheorem{cor}[thm]{Corollary}

\newcommand{\bea}{\begin{eqnarray}}
\newcommand{\eea}{\end{eqnarray}}
\newcommand{\beaa}{\begin{eqnarray*}}
\newcommand{\eeaa}{\end{eqnarray*}}
\newcommand{\be}{\begin{equation}}
\newcommand{\ee}{\end{equation}}
\newcommand{\nn}{\nonumber}

\begin{document}

\title{3D Bosons, 3-Jack polynomials and affine Yangian of ${\mathfrak{gl}}(1)$}
\author{Wang Na\dag\footnote{Corresponding author: wangna@henu.edu.cn },\ Wu Ke\ddag\\
\dag\small School of Mathematics and Statistics, Henan University, Kaifeng, 475001, China\\
\ddag\small School of Mathematical Sciences, Capital Normal University, Beijing 100048, China}

\date{}
\maketitle

\begin{abstract}
3D (3 dimensional) Young diagrams are a generalization of 2D Young diagrams. In this paper, We consider 3D Bosons and 3-Jack polynomials. We associate three parameters $h_1,h_2,h_3$ to $y,x,z$-axis respectively. 3-Jack polynomials are polynomials of $P_{n,j},
 n\geq j$ with coefficients in $\C(h_1,h_2,h_3)$, which are the generalization of Schur functions and Jack polynomials to 3D case. Similar to Schur functions, 3-Jack polynomials can also be determined by the vertex operators and the Pieri formulas.

\end{abstract}

\section{Introduction}\label{sec1}
2D Young diagrams and symmetric functions are attractive research objects, which were used to determine irreducible characters of highest weight representations of the classical groups\cite{FH,Mac,weyl}. Recently they appear in mathematical physics, especially in integrable models. In \cite{MJD}, the group in the  Kyoto school uses Schur functions in a remarkable way to understand the KP and KdV hierarchies.  In \cite{NVT,PS}, Tsilevich and Su\l kowski, respectively, give the realization of the phase model in the algebra of Schur functions and build the relations between the $q$-boson model and Hall-Littlewood functions. In \cite{wang}, we build the relations between the statistical models, such as phase model, and KP hierarchy by using 2D Young diagrams and Schur functions. In \cite{WLZZ}, the authors show that the states in the $\beta$-deformed Hurwitz-Kontsevich matrix model can be represented as the Jack polynomials.

3D Young diagrams (plane partitions) are a generalization of 2D Young diagrams, which arose naturally in crystal melting model\cite{ORV,NT}. 3D Young diagrams also have many applications in many fields of mathematics and physics, such as statistical models, number theory, representations of some algebras (Ding-Iohara-Miki algebras, affine Yangian, etc). In this paper, we consider 3D Bosons and 3D symmetric functions.

The Schur functions $S_\lambda$ defined on 2D Young diagrams $\lambda$ can be determined by the vertex operator and the Jacobi-Trudi formula. Let ${ p}=(p_1,p_2,\cdots)$.
The operators $S_n({ p})$ are determined by the vertex operator:
\be\label{hxi}
e^{\xi({ p},z)}=\sum_{n=0}^\infty S_n({ p})z^n,\quad \text{with}\ \  \xi({ p},z)=\sum_{n=1}^\infty \frac{p_n}{n} z^n
\ee
and set $S_n({p})=0$ for $n<0$.
Note that $S_n({ p})$ is the complete homogeneous symmetric function by the Miwa transform, i.e., replacing $p_i$ with the power sum $\sum_k x_k^i$.
For 2D Young diagram $\lambda=(\lambda_1,\lambda_2,\cdots,\lambda_l)$, the Schur function $S_{\lambda}=S_{\lambda}({ p})$ is a polynomial of variables ${ p}$ in $\C[{ p}]$ defined by the Jacobi-Trudi formula \cite{FH,Mac}:
\be\label{slambda}
S_{\lambda}({ p})=\text{det}\left(
h_{\lambda_{i}-i+j}({ p})
 \right)_{1\leq i,j\leq l}.
\ee
The Jacobi-Trudi formula can be replaced by the Pieri formula\cite{FH}
\be
S_{n}S_\lambda=\sum_{\mu}C_{n,\lambda}^\mu S_{\mu}.
\ee

The 3-Jack polynomials, which are symmetric functions defined on 3D Young diagrams, can also be determined this way. We associate three parameters $h_1,h_2,h_3$ to $y,x,z$-axis respectively, where $h_1,h_2,h_3$ are the parameters appeared in affine Yangian of $\mathfrak{gl}(1)$. 3-Jack polynomials are symmetric about three coordinate axes, which means they are symmetric about $h_1,h_2,h_3$. Let $(n)$ be the 2D Young diagrams of $n$ box along $y$-axis, which is treated as 3D Young diagram which have one layer in $z$-axis direction, then 3-Jack polynomials $\tilde{J}_n$ can be determined by the vertex operator:
\be\label{exiy}
e^{\xi_y(P,z)}=\sum_{n\geq 0}\frac{1}{\langle \tilde{J}_n,\tilde{J}_n\rangle h_1^n}\tilde{J}_n(P)z^n,\ee
 with
  \be
  \xi_y(P,z)=\sum_{n,j=1}^\infty \frac{ P_{n,j}}{j!\left(\begin{array}{cc}{n+j-1}\\{n-j}\end{array}\right)}\frac{d_{n,j}}{h_1^j}\prod_{k=1}^{j-1}\frac{1}{k+h_2h_3}z^n,
\ee
where
\be
d_{n,j}=\begin{cases}
 1 & \text{ if } \ n=j, \\
  j& \text{ if}\ n>j.
\end{cases}
\ee
When $P_{n,1}=p_1,\ P_{n,j\geq 1}=0$, the vertex operator (\ref{exiy}) becomes (\ref{hxi}). The 3-Jack polynomials of $n$ boxes along $x$-axis or $z$-axis can also expressed this way from the symmetry. In fact, for 2D Young diagrams $\lambda$, which are treated as the 3D Young diagrams who have one layer, 3-Jack polynomials $\tilde{J}_\lambda$ can also be determined by the vertex operators. 3-Jack polynomials of 3D Young diagrams who have more than one layer can be determined by the ``Pieri formula" $\tilde{J}_\lambda\tilde{J}_\pi$ with $\pi$ being 3D Young diagrams. This Pieri formula can be determined by the representation of affine Yangian of $\mathfrak{gl}(1)$ on 3D Young diagrams.

The paper is organized as follows. In section \ref{sect2}, we recall the definition of affine Yangian of $\mathfrak{gl}(1)$ and its representation on 3D Young diagrams. In section \ref{sect3},  we consider 3D Bosons and the algebra $W$. In section \ref{sect4},  we discuss the realizations of 3D Bosons and the operators in the algebra $W$ by using affine Yangian of $\mathfrak{gl}(1)$ and its representation on 3D Young diagrams.
In section \ref{sect5}, we show the expressions of 3-Jack polynomials by the vertex operators and the Pieri formulas.
\section{Affine Yangian of $\mathfrak{gl}(1)$}\label{sect2}
In this section, we recall the definition of the affine Yangian of $\mathfrak{gl}(1)$ as in papers \cite{3-schur,Wang2,Pro, Tsy} first. Then we calculate some properties of affine Yangian which have relations with 3D Bosons.
The affine Yangian $\mathcal{Y}$ of $\mathfrak{gl}(1)$ is an associative algebra with generators $e_j, f_j$ and $\psi_j$, $j = 0, 1, \ldots$ and the following relations\cite{Pro, Tsy}
\begin{eqnarray}
&&\left[ \psi_j, \psi_k \right] = 0,\\
&&\left[ e_{j+3}, e_k \right] - 3 \left[ e_{j+2}, e_{k+1} \right] + 3\left[ e_{j+1}, e_{k+2} \right] - \left[ e_j, e_{k+3} \right]\nonumber \\
&& \quad + \sigma_2 \left[ e_{j+1}, e_k \right] - \sigma_2 \left[ e_j, e_{k+1} \right] - \sigma_3 \left\{ e_j, e_k \right\} =0,\label{yangian1}\\
&&\left[ f_{j+3}, f_k \right] - 3 \left[ f_{j+2}, f_{k+1} \right] + 3\left[ f_{j+1}, f_{k+2} \right] - \left[ f_j, f_{k+3} \right] \nonumber\\
&& \quad + \sigma_2 \left[ f_{j+1}, f_k \right] - \sigma_2 \left[ f_j, f_{k+1} \right] + \sigma_3 \left\{ f_j, f_k \right\} =0, \label{yangian2}\\
&&\left[ e_j, f_k \right] = \psi_{j+k},\label{yangian3}\\
&& \left[ \psi_{j+3}, e_k \right] - 3 \left[ \psi_{j+2}, e_{k+1} \right] + 3\left[ \psi_{j+1}, e_{k+2} \right] - \left[ \psi_j, e_{k+3} \right]\nonumber \\
&& \quad + \sigma_2 \left[ \psi_{j+1}, e_k \right] - \sigma_2 \left[ \psi_j, e_{k+1} \right] - \sigma_3 \left\{ \psi_j, e_k \right\} =0,\label{yangian4}\\
&& \left[ \psi_{j+3}, f_k \right] - 3 \left[ \psi_{j+2}, f_{k+1} \right] + 3\left[ \psi_{j+1}, f_{k+2} \right] - \left[ \psi_j, f_{k+3} \right] \nonumber\\
&& \quad + \sigma_2 \left[ \psi_{j+1}, f_k \right] - \sigma_2 \left[ \psi_j, f_{k+1} \right] + \sigma_3 \left\{ \psi_j, f_k \right\} =0,\label{yangian5}
\end{eqnarray}
together with boundary conditions
\begin{eqnarray}
&&\left[ \psi_0, e_j \right]  = 0, \left[ \psi_1, e_j \right] = 0,  \left[ \psi_2, e_j \right]  = 2 e_j ,\label{yangian6}\\
&&\left[ \psi_0, f_j \right]  = 0,  \left[ \psi_1, f_j \right]  = 0,  \left[ \psi_2, f_j \right]  = -2f_j ,\label{yangian7}
\end{eqnarray}
and a generalization of Serre relations
\begin{eqnarray}
&&\mathrm{Sym}_{(j_1,j_2,j_3)} \left[ e_{j_1}, \left[ e_{j_2}, e_{j_3+1} \right] \right]  = 0, \label{yangian8} \\
&&\mathrm{Sym}_{(j_1,j_2,j_3)} \left[ f_{j_1}, \left[ f_{j_2}, f_{j_3+1} \right] \right]  = 0,\label{yangian9}
\end{eqnarray}
where $\mathrm{Sym}$ is the complete symmetrization over all indicated indices which include $6$ terms. In this paper, we set $\psi_0=1$ with no loss of generality.

The notations $\sigma_2,\ \sigma_3$ in the definition of affine Yangian are functions of three complex numbers $h_1, h_2$ and $h_3$:
\beaa
\sigma_1 &=& h_1+h_2+h_3=0,\\
\sigma_2 &=& h_1 h_2 + h_1 h_3 + h_2 h_3,\\
 \sigma_3 &=& h_1 h_2 h_3.
\eeaa

This affine Yangian has the representation on 3D Young diagrams or Plane partitions. A plane partition $\pi$ is a 2D Young diagram in the first quadrant of plane $xOy$ filled with non-negative integers that form nonincreasing rows and columns  \cite{FW, ORV}. The number in the position $(i,j)$ is denoted by $\pi_{i,j}$
\[
\left( \begin{array}{ccc}
\pi_{1,1} & \pi_{1,2} &\cdots \\
\pi_{2,1} & \pi_{2,2} &\cdots \\
\cdots & \cdots &\cdots \\
\end{array} \right).
\]
 The integers $\pi_{i,j}$ satisfy
\[
\pi_{i,j}\geq \pi_{i+1,j},\quad \pi_{i,j}\geq \pi_{i,j+1},\quad \lim_{i\rightarrow \infty}\pi_{i,j}=\lim_{j\rightarrow \infty}\pi_{i,j}=0
\]
for all integers $i,j\geq 0$. Piling $\pi_{i,j}$ cubes over position $(i, j)$ gives a 3D Young diagram. 3D Young diagrams arose naturally in the melting crystal model\cite{ORV,NT}. We always identify 3D Young diagrams with plane
partitions as explained above. For example, the 3D Young diagram
$\twoboxy$
can also be denoted by the plane partition
$(1,1)$.

 As in our paper \cite{3DFermionYangian}, we use the following notations. For a 3D Young diagram $\pi$, the notation $\Box\in \pi^+$ means that this box is not in $\pi$ and can be added to $\pi$. Here ``can be added'' means that when this box is added, it is still a 3D Young diagram. The notation $\Box\in \pi^-$ means that this box is in $\pi$ and can be removed from $\pi$. Here ``can be removed" means that when this box is removed, it is still a 3D Young diagram. For a box $\Box$, we let
\begin{equation}\label{epsilonbox}
h_\Box=h_1y_\Box+h_2x_\Box+h_3z_\Box,
\end{equation}
where $(x_\Box,y_\Box,z_\Box)$ is the coordinate of box $\Box$ in coordinate system $O-xyz$. Here we use the order $y_\Box,x_\Box,z_\Box$ to match that in paper \cite{Pro}.

Following \cite{Pro,Tsy}, we introduce the generating functions:
\begin{eqnarray}
e(u)&=&\sum_{j=0}^{\infty} \frac{e_j}{u^{j+1}},\nonumber\\
f(u)&=&\sum_{j=0}^{\infty} \frac{f_j}{u^{j+1}},\\
\psi(u)&=& 1 + \sigma_3 \sum_{j=0}^{\infty} \frac{\psi_j}{u^{j+1}},\nonumber
\end{eqnarray}
where $u$ is a parameter.
Introduce
\begin{equation}\label{psi0}
\psi_0(u)=\frac{u+\sigma_3\psi_0}{u}
\end{equation}
and
\begin{eqnarray} \label{dfnvarphi}
\varphi(u)=\frac{(u+h_1)(u+h_2)(u+h_3)}{(u-h_1)(u-h_2)(u-h_3)}.
\end{eqnarray}
For a 3D Young diagram $\pi$, define $\psi_\pi(u)$ by
\begin{eqnarray}\label{psipiu}
\psi_\pi(u)=\psi_0(u)\prod_{\Box\in\pi} \varphi(u-h_\Box).
\end{eqnarray}
In the following, we recall the representation of the affine Yangian on 3D Young diagrams as in paper \cite{Pro} by making a slight change. The representation of affine Yangian on 3D Young diagrams is given by
\begin{eqnarray}
\psi(u)|\pi\rangle&=&\psi_\pi(u)|\pi\rangle,\\
e(u)|\pi\rangle&=&\sum_{\Box\in \pi^+}\frac{E(\pi\rightarrow\pi+\Box)}{u-h_\Box}|\pi+\Box\rangle,\label{eupi}\\
f(u)|\pi\rangle&=&\sum_{\Box\in \pi^-}\frac{F(\pi\rightarrow\pi-\Box)}{u-h_\Box}|\pi-\Box\rangle\label{fupi}
\end{eqnarray}
where $|\pi\rangle$ means the state characterized by the 3D Young diagram $\pi$ and the coefficients
\begin{equation}\label{efpi}
E(\pi \rightarrow \pi+\square)=-F(\pi+\square \rightarrow \pi)=\sqrt{\frac{1}{\sigma_3} \operatorname{res}_{u \rightarrow h_{\square}} \psi_\pi(u)}
\end{equation}
 Equations (\ref{eupi}) and (\ref{fupi}) mean generators $e_j,\ f_j$ acting on the 3D Young diagram $\pi$ by
\begin{equation}
\begin{aligned}
e_j|\pi\rangle &=\sum_{\square \in \pi^{+}} h_{\square}^j E(\pi \rightarrow \pi+\square)|\pi+\square\rangle,
\end{aligned}
\end{equation}
\begin{equation}
\begin{aligned}
f_j|\pi\rangle &=\sum h_{\square}^j F(\pi \rightarrow \pi-\square)|\pi-\square\rangle .
\end{aligned}
\end{equation}

 The triangular decomposition of affine Yangian $\mathcal{Y}$ is
  \be
 \mathcal{Y}=\mathcal{Y}^+\oplus\mathcal{B}\oplus\mathcal{Y}^-
  \ee
where $\mathcal{Y}^+$ is the subalgebra generated by generators $e_j$ with relations (\ref{yangian1}) and (\ref{yangian8}), $\mathcal{B}$ is the commutative subalgebra with generators $\psi_j$,  $\mathcal{Y}^-$ is the subalgebra generated by generators $f_j$ with relations (\ref{yangian2}) and (\ref{yangian9}).

Define the anti-automorphism $\tilde{a}$ by
\be
\tilde{a}(e_j)=-f_j
\ee
The quadratic form  on $\mathcal{Y}^+|0\rangle$ is defined by
\bea\label{quadraticform}
\tilde {B}(x|0\rangle, y|0\rangle)=\langle 0|\tilde a (y)x|0\rangle
\eea
where $x, y\in\mathcal{Y}^+$. Note that the quadratic form here is different from the Shapovalov form in \cite{Pro}.
For 3D Young diagrams $\pi,\pi'$ and let $\pi=x|0\rangle, \pi'=y|0\rangle$ for $x, y\in\mathcal{Y}^+$, define
the orthogonality
\be\label{orthogonality}
\langle\pi',\pi\rangle=\langle 0|\tilde a (y)x|0\rangle.
\ee

In the following, we calculate some properties of affine Yangian which have close relation with 3D Bosons. We want to have the fact that 3D Bosons will become 2D Bosons when we require $a_{n,k}=0, k>1$ and $a_{n,1}=a_n$.  As in paper \cite{Pro}, when $n>0$, $a_{n,1}$ and $a_{-n,1}$ are defined to be
\bea\label{2Dboson}
a_{n,1}:=-\frac{1}{(n-1)!}\text{ad}_{f_1}^{n-1} f_0,\\
a_{-n,1}:=\frac{1}{(n-1)!}\text{ad}_{e_1}^{n-1} e_0.
\eea

The set
\[
\{\text{ad}_{e_1}^{n-1} e_0, n=1,2,3,\cdots\}
\]
is denoted by $\diamondsuit$. That the operator $A$ commutes with $\diamondsuit$ means $A$ commute with every element in $\diamondsuit$.
\begin{prp}\label{communicationAEA}
If the operator $A$ commutes with $\diamondsuit$, so does $[e_1, A]$.
\end{prp}
\begin{proof}
That $A$ commutes with $\diamondsuit$ means $[A,\text{ad}_{e_1}^{n-1} e_0]=0$ for $n=1,2,3,\cdots$. By Jacobi identity,
\beaa
[[e_1, A],\text{ad}_{e_1}^{n-1} e_0]=-[[A,\text{ad}_{e_1}^{n-1} e_0], e_1]-[[\text{ad}_{e_1}^{n-1} e_0, e_1], A]=0.
\eeaa
\end{proof}
Then we get the following result.
\begin{prp}
If the operator $A$ commutes with $\diamondsuit$, then we get that $\text{ad}_{e_1}^{n-1} A$ for $n=1,2,3,\cdots$ commute with $\diamondsuit$.
\end{prp}
By similar calculation, we can get that if the operator $B$ commutes with $\text{ad}_{e_1}^{n-1} e_0$ and $\text{ad}_{e_1}^{n-1} A$ for $n=1,2,3,\cdots$, so do $\text{ad}_{e_1}^{n-1} B$ for $n=1,2,3,\cdots$.

\section{3D Bosons}\label{sect3}
Introduce the space of all polynomials
\[
\C[p]=\C[p_1,p_{2,1},p_{2,2},p_{3,1},p_{3,2},p_{3,3},\cdots].
\]
Every polynomial is a function of infinitely many variables $$p=(p_1,p_{2,1},p_{2,2},p_{3,1},p_{3,2},p_{3,3},\cdots),$$
but each polynomial itself is a finite sum of monomials, so involves only finitely many of the variables.

Define the weight of $p_{n,k}$ to be $n$, and the weight of monomial $p_{n_1,k_1}^{l_1}\cdots p_{n_s,k_s}^{l_s}$ to be $l_1 n_1+\cdots +l_s n_s$, then the 3D Bosonic Fock space is written into
\be
\C[p]=\bigoplus_{n=0}^\infty \C[p]_n
\ee
where $\C[p]_n$ is the space of polynomials of weight $n$, which is a subspace of $\C[p]$. The basis of $\C[p]_0$ is $1$, the basis of $\C[p]_1$ is $p_1$, the basis of $\C[p]_2$ is $p_1^2, p_{2,1}, p_{2,2}$, the basis of $\C[p]_3$ is $p_1^3, p_1p_{2,1}, p_1p_{2,2}, p_{3,1}, p_{3,2}, p_{3,3}$, we can write the basis of every subspace $\C[p]_n$.  By the results in \cite{YZ}, we know that the elements in the basis of $\C[p]_n$ is in one to one correspondence with 3D Young diagrams of total box number $n$. Then we have
\be
\sum_{n=0}^\infty \text{dim} (\C[p]_n)q^n=\prod_{n=1}^\infty\frac{1}{(1-q^n)^n}
\ee
where the notation $\text{dim} (\C[p]_n)$ means the dimension of the vector space $\C[p]_n$.

Define the form $\langle \cdot,\cdot\rangle$ on $\C[p]$ by $\langle p_{n,j},p_{m,i}\rangle=0$ unless $n=m,\ i=j$. When $i=j=1$, $\langle p_{n,1},p_{m,1}\rangle$ equal $\langle p_{n},p_{m}\rangle$ in the 2D Bosons. Define
\be
\langle p_{j,j},p_{j,j}\rangle=(-1)^{j-1} j!\prod_{k=1}^{j-1}(k^3+k^2\sigma_2+\sigma_3^2)
\ee
and
\be
\langle p_{n,j},p_{n,j}\rangle=\left(\begin{array}{cc}n+j-1\\n-j\end{array}\right)\langle p_{j,j},p_{j,j}\rangle.
\ee

Consider operators $b_{n,k}$, where $n\in\Z, n\neq 0$, and $k\in \Z, 0<k\leq |n|$, with the commutation relations
\be\label{3dbosonrel}
[b_{m,l},b_{n,k}]=m\delta_{m+n,0}\delta_{l,k}\langle p_{l,|m|},p_{l,|m|}\rangle.
\ee
The operators $b_{n,k}$ with the relations above are called 3D Boson, and the algebra generated by $b_{n,k}$ with relations (\ref{3dbosonrel}) is called 3D Heisenberg algebra, which is denoted by $\mathfrak{B}$. Using the commutation relations (\ref{3dbosonrel}), we can see that any element in 3D Heisenberg algebra can be expressed in a unique way as a linear combination of the following elements:
\[
b_{-m_1,l_1}^{\alpha_1}\cdots b_{-m_s,l_s}^{\alpha_s}b_{n_1,k_1}^{\beta_1}\cdots b_{n_t,k_t}^{\beta_t}
\]
for
\[
(-m_1,l_1)<\cdots <(-m_s,l_s),\ \ (n_1,k_1)<\cdots<(n_t,k_t),\ \ \alpha_i,\beta_i=1,2,\cdots,
\]
where the notation $(m,l)<(n,k)$ means $m<n$ or $m=n, l<k$.

Define a linear map $\rho:\mathfrak{B}\rightarrow \text{End}(\C[p])$ by
\bea
\rho (b_{-n,k})= p_{n,k},\ \ \ \rho (b_{n,k})=\langle p_{l,|m|},p_{l,|m|}\rangle n\frac{\partial}{ \partial p_{n,k}},
\eea for $n>0$,
which gives a representation of the Heisenberg algebra $\mathfrak{B}$ on polynomial space $\C[p]$. The representation space $\C[p]$ is called the 3D Bosonic Fock space. We call the operators $b_{n,k}$ annihilation operators and $b_{-n,k}$ creation operators for $n>0$. From the commutation relations (\ref{3dbosonrel}), it is easy to find that all the creation operators commute among themselves, so do all the annihilation operators. The element $1\in\C[p]$ is called the vacuum state. Every annihilation operators kill the vacuum state, that is, $n\frac{\partial}{ \partial p_{n,k}}1=0$. The 3D Bosonic Fock space is generated by the vacuum state:
\be
\C[p]=\mathfrak{B}\cdot 1:=\{b\cdot 1|b\in \mathfrak{B}\}.
\ee

We give a remark to explain why we use ``3D" in the 3D Boson and 3D Heisenberg algebra. 3D is used to distinguish 3D Boson and ordinary Boson, 3D Heisenberg algebra and ordinary Heisenberg algebra. Similar to that the ordinary Bosonic Fock space is isomorphic to the space of 2D Young diagrams, the 3D Bosonic Fock space is isomorphic to the space of 3D Young diagrams, this is the reason we use the notation 3D.

In the following, we consider the algebra $W$ with the generators $a_{n,k},$  ($n\in\Z, n\neq 0$, and $k\in \Z, 0<k\leq |n|$). We let $a_{n,1}=b_{n,1}$ and $a_{-2,2}=b_{-2,2}$. The commutation relations in $W$ are
\bea
[a_{n,1},a_{m,1}]&=&\psi_0 n\delta_{n+m,0},
\eea
\bea
[a_{n,1}, a_{m,j}]&=& 0,\ \ \text{when} \ \ j>1,
\eea
and when $j>1,\ k>1$,
\bea\label{anjank}
[a_{m,j},a_{n,k}]=\sum\limits_{\substack{0\leq l\leq j+k-2 \\ j+k-l\text{even}}}\frac{j!k!}{l!}C_{jk}^lN_{jk}^l(m,n)a_{m+n,l},
\eea
where the coefficients $N_{jk}^l(m,n)$ are
\bea
N_{jk}^0(m,n)&=&\left(\begin{array}{cc}m+j-1\\j+k-1\end{array}\right)\delta_{m+n,0},\\
N_{jk}^l(m,n)&=&\sum_{s=0}^{j+k-l-1}\frac{(-1)^s}{(j+k-l-1)!(2l)_{j+k-l-1}}\left(\begin{array}{cc}j+k-l-1\\s\end{array}\right)\times\nn\\
&&[j+m-1]_{j+k-l-1-s}[j-m-1]_{s}[k+n-1]_{s}[k-n-1]_{j+k-l-1-s},\nn
\eea
and the structure constants $C_{jk}^l$ are
\bea
C_{jk}^0&=&\frac{(j-1)!^2(2j-1)!}{4^{j-1}(2j-1)!!(2j-3)!!}\delta_{jk}c_{j},\\
C_{jk}^l&=&\frac{1}{2\times 4^{j+k-l-2}}(2l)_{j+k-l-1}\times {}_4F_3\left(\begin{array}{cc}\frac{1}{2},\frac{1}{2},-\frac{1}{2}(j+k-l-2),-\frac{1}{2}(j+k-l-1)\\ \frac{3}{2}-j,\frac{3}{2}-k,\frac{1}{2}+l\end{array};1\right),\nn
\eea
with
\bea
(a)_n&=&a(a+1)\cdots (a+n-1),\\
{[a]}_n&=&a(a-1)\cdots (a-n+1),\\
{}_mF_n\left(\begin{array}{cc}a_1,\cdots,a_m\\ b_1,\cdots,b_n\end{array};z\right)&=&\sum_{k=0}^\infty\frac{(a_1)_k\cdots (a_m)_k}{(b_1)_k\cdots (b_n)_k}\frac{z^k}{k!}.
\eea
We give a remark here to explain the central charges. The central charge of $a_{n,1}$ is $\psi_0$, we always take $\psi_0=1$ without loss the symmetry of affine Yangian of $\mathfrak{gl}(1)$ about three coordinate axes. The central charge $c_j$ of $a_{n,j}$ is dependent on $j$. The first few of them are
\beaa
c_{2}&=& -2 (1+\sigma_2+\sigma_3^2),\\
c_{3}&=& 6 (1+\sigma_2+\sigma_3^2)(8+4\sigma_2+\sigma_3^2),\\
c_4&=& -144 \frac{(1+\sigma_2+\sigma_3^2)(8+4\sigma_2+\sigma_3^2)(27+9\sigma_2+\sigma_3^2)(-1+\sigma_2+\sigma_3^2)}{(-17+5\sigma_2+5\sigma_3^2)}.
\eeaa
Define
\[
\kappa(n)=-(n+h_1h_2)(n+h_1h_3)(n+h_2h_3)=-(n^3+n^2\sigma_2+\sigma_3^2),
\]
then
\[
c_2=2\kappa(1),\ c_{3}=6\kappa(1)\kappa(2),
\]
which match $\langle P_{2,2},P_{2,2}\rangle$ and $\langle P_{3,3},P_{3,3}\rangle$ in \cite{3-schur} respectively, and
\[
c_4=144\frac{ \kappa(1)\kappa(2)\kappa(3)\kappa(-1)}{5\kappa(1)+22},
\]
which matches $\langle P_{4,4},P_{4,4}\rangle$ in \cite{WLin}. Here we choose $P_{4,4}=E_{13}|0\rangle$ in \cite{WLin}.

For $n>0$, define $P_{n,j}$ be the representation of $a_{-n,j}$ on 3D Young diagrams,
then the 3-Jack polynomials as vectors in the representation space are functions of $P_{n,j}$. Note that the set $\{P_{n,j}\}$ are related with $\{p_{n,j}\}$. For example, $P_{n,1}=p_{n,1}$ and $P_{2,2}=p_{2,2}$.
\section{3D Bosons in affine Yangian of ${\mathfrak{gl}}(1)$}\label{sect4}
We denote $[e_i,e_j]$ by $e_{i,j}$. We have the recurrence relation: when $j+k=2n$ in (\ref{yangian1}),
\beaa
e_{n+2,n+1}&=&\frac{1}{2n+3}(e_{2n+3,0}+\sigma_2\sum_{j=0}^n(j+1)e_{2n+1-j,j}-\\
&&\sigma_3\sum_{j=0}^{n-1}(1+2+\cdots+(j+1))\{e_{2n-j},e_j\}-\\
&&(1+2+\cdots+(n+1)\sigma_3e_ne_n)),
\eeaa
the right hand side is denoted by $\blacktriangle$, then
\beaa
e_{n+3+i,n-i}&=&(2i+3)\blacktriangle+(-\sigma_2)\sum_{j=0}^i(j+1)e_{n+1+i-j,n-i+j}+\\
&&\sigma_3\sum_{j=0}^{i-1}(1+2+\cdots+(j+1))\{e_{n+i-j},e_{n-i+j}\}+\\
&&(1+2+\cdots+(i+1)\sigma_3e_ne_n))
\eeaa
where $i=0,1,\cdots,n-1$.
 When $j+k=2n+1$ in (\ref{yangian1}),
\beaa
e_{n+3,n+1}&=&\frac{1}{n+2}(e_{2n+4,0}+\sigma_2\sum_{j=0}^n(j+1)e_{2n+2-j,j}-\\
&&\sigma_3\sum_{j=0}^{n}(1+2+\cdots+(j+1))\{e_{2n+1-j},e_j\},
\eeaa
the right hand side is denoted by $\blacktriangledown$, then
\beaa
e_{n+4+i,n-i}&=&(i+2)\blacktriangledown+(-\sigma_2)\sum_{j=0}^i(j+1)e_{n+2+i-j,n-i+j}+\\
&&\sigma_3\sum_{j=0}^{i}(1+2+\cdots+(j+1))\{e_{n+1+i-j},e_{n-i+j}\},
\eeaa
where $i=0,1,\cdots,n-1$.

We rewrite the recurrence relation in the matrix form
\begin{prp}When $j+k=2n$ in (\ref{yangian1}),
\bea
\left(\begin{array}{cccc} e_{2n+2,1}\\e_{2n+1,2}\\e_{2n,3}\\\cdots\\e_{n+3,n}\\e_{n+2,n+1}\end{array}\right)&=&-\sigma_2 (A_{n+1}^{-1}(I_{n+1}-J_{n+1}))_{\widehat{j_1}}\left(\begin{array}{cccc} e_{2n,1}\\e_{2n-1,2}\\e_{2n-2,3}\\\cdots\\e_{n+2,n-1}\\e_{n+1,n}\end{array}\right)+\sigma_3(A_{n+1}^{-1})_{\widehat{j_1,j_{n+1}}}\left(\begin{array}{cccc} e_{2n-1,1}\\e_{2n-2,2}\\e_{2n-3,3}\\\cdots\\e_{n+2,n-2}\\e_{n+1,n-1}\end{array}\right)\nn\\
&&+\frac{e_{2n+3,0}+\sigma_2e_{2n+1,0}-\sigma_3e_{2n,0}}{2n+3}\left(\begin{array}{cccc} 2n+1\\2n-1\\\cdots\\5\\3\\1\end{array}\right)+2\sigma_3 A_{n+1}^{-1}\left(\begin{array}{cccc} e_0e_{2n}\\e_1e_{2n-1}\\e_2e_{2n-2}\\ \cdots\\e_{n-1}e_{n+1}\\\frac{1}{2}e_ne_{n}\end{array}\right)
\eea

When $j+k=2n+1$ in (\ref{yangian1}),
\bea
&&\left(\begin{array}{cccc} e_{2n+3,1}\\e_{2n+2,2}\\e_{2n+1,3}\\\cdots\\e_{n+4,n}\\e_{n+3,n+1}\end{array}\right)=-\sigma_2 (B_{n+1}^{-1}(I_{n+1}-J_{n+1}))_{\widehat{j_1}}\left(\begin{array}{cccc} e_{2n,1}\\e_{2n-1,2}\\e_{2n-2,3}\\\cdots\\e_{n+2,n-1}\\e_{n+1,n}\end{array}\right)+\sigma_3(B_{n+1}^{-1})_{\widehat{j_1}}\left(\begin{array}{cccc} e_{2n-1,1}\\e_{2n-2,2}\\e_{2n-3,3}\\\cdots\\e_{n+2,n-2}\\e_{n+1,n-1}\end{array}\right)\nn\\
&&\ \ \ \ \ \ \ \ \ \ +\frac{e_{2n+4,0}+\sigma_2e_{2n+2,0}-\sigma_3e_{2n+1,0}}{n+2}\left(\begin{array}{cccc} n+1\\n\\\cdots\\3\\2\\1\end{array}\right)+2\sigma_3 B_{n+1}^{-1}\left(\begin{array}{cccc} e_0e_{2n+1}\\e_1e_{2n}\\e_2e_{2n-1}\\ \cdots\\e_{n-1}e_{n+2}\\e_ne_{n+1}\end{array}\right)
\eea
\end{prp}
\begin{cor}
The vector space spanned by $e_ie_j$ for $i,j=0,1,2,\cdots$, where $e_i$ are the generators of affine Yangian, has a basis
\be
e_{n,0}, \ e_me_n \ ( m\leq n)
\ee
for $m,n=0,1,2,\cdots$.
\end{cor}

Suppose the vector
\[
\sum_{n\geq 0}d_ne_{n,0}+\sum_{m\leq n}b_{m n}e_me_n
\]
with the coefficients $d_n,\ b_{mn}$, be commutative with Bosons $a_{n,1}$, we only find the zero coefficients, that is, $d_n=0$ and $b_{mn}$=0. But we found the vector
\be
[e_2,e_0]-\sigma_3 [e_1,e_0]-e_0^2-2\sum_{n=1}^\infty a_{-(n+2),1}a_{n,1}
\ee
commute with Bosons $a_{n,1}$, which is denoted by $a_{-2,2}$ or $p_{2,2}$ in the representation on 3D Young diagrams. The communication $[a_{-2,2},a_{-n,1}]=0,\ n>0$ holds since
\beaa
&&[[e_2,e_0]-\sigma_3 [e_1,e_0], a_{-n,1}]=2na_{-{(n+2)},1},\\
&&[\sum_{m=1}^\infty a_{-(m+2),1}a_{m,1}, a_{-n,1}]=na_{-{(n+2)},1}.
\eeaa
The communication $[a_{-2,2},a_{n,1}]=0, n>0$ holds since
\beaa
&&[[e_2,e_0]-\sigma_3 [e_1,e_0], a_{-n,1}]=-2na_{{(n-2)},1},\\
&&[\sum_{m=1}^\infty a_{-(m+2),1}a_{m,1}, a_{n,1}]=-na_{{(n-2)},1}.
\eeaa
For $n\geq 2$, define
\be
a_{-(n+1),2}=\frac{1}{(n-1)!}\text{ad}_{e_1}^{n-1}\left([e_2,e_0]-\sigma_3[e_1,e_0]\right)-\sum_{i+j=-(n+1)}:a_{i,1}a_{j,1}:.
\ee
Similarly, define
\be
a_{n+1,2}=-\frac{1}{(n-1)!}\text{ad}_{f_1}^{n-1}\left([f_2,f_0]-\sigma_3[f_1,f_0]\right)-\sum_{i+j=(n+1)}:a_{i,1}a_{j,1}:.
\ee
We denote the first term in $a_{n,2}$ by $2L_n$, and the second term by $2\bar{L}_{n}$. Then we have that $L_n$ and $\bar{L}_n$ separately satisfy the virasoro relations, that is,
\bea
[L_{m}, L_{n}]&=&(m-n) L_{m+n}+\frac{m^3-m}{12} \delta_{m+n,0}(-\sigma_2-\sigma_3^2),
\eea
and
\bea
[\bar{L}_{m}, \bar{L}_{n}]&=&(m-n) \bar{L}_{m+n}+\frac{m^3-m}{12} \delta_{m+n,0}.
\eea

By calculation, we can prove the following relations hold
\begin{prp}For $n\geq 2$,
\be\label{e1barL-}
\frac{1}{n-1}e_1\bar{L}_{-n}|0\rangle=\bar{L}_{-(n+1)}|0\rangle.
\ee
\end{prp}
This result means
\begin{prp}For $n\geq 2$,
\be
a_{-n,2}|0\rangle=P_{n,2}|0\rangle=\frac{1}{(n-2)!}\underbrace{e_1\cdots e_1}_{n-2}P_{2,2}|0\rangle.
\ee
\end{prp}

From (\ref{yangian6}) and (\ref{yangian7}), we have
\be
[\psi_2, a_{n,1}]=-2n a_{n,1}
\ee
for any $n$. Then we have
\begin{lem}
For any $n\neq 0$,
\be
[e_1,a_{n,1}]=-na_{n-1,1}.
\ee
\end{lem}
\begin{proof}
When $n<0$, it holds clearly. When $n>0$, it can be proved by induction.
\end{proof}
\begin{prp}
For $n\geq 1$,
\be
\frac{1}{(n-1)}[e_1,\bar{L}_{-n}]=\bar{L}_{-(n+1)}.
\ee
\end{prp}
\begin{proof}
Since the relation (\ref{e1barL-}) hold, we only need to prove
\[
\frac{1}{(n-1)}[e_1,(\bar{L}_{-n})_+]=(\bar{L}_{-(n+1)})_+,
\]
where $(\bar{L}_{-n})_+$ include all terms which have annihilation operators in $\bar{L}_{-n}$, that is,
\[
(\bar{L}_{-n})_+=\sum_{k\geq 1}a_{-(n+k)}a_{k}.
\]
Then
\beaa
[e_1,(\bar{L}_{-n})_+]&=&\sum_{k\geq 1}([e_1,a_{-(n+k)}]a_{k}+a_{-(n+k)}[e_1,a_{k}])\\
&=&(n-1)\sum_{k\geq 1}a_{-(n+k+1)}a_{k}=(n-1)(\bar{L}_{-(n+1)})_+,
\eeaa
which means the result holds.
\end{proof}
Note that here
\[
e_1=L_{-1}\neq \bar{L}_{-1}.
\]
\begin{prp}
For $n\geq 2$,
\bea
a_{-n,2}&=&\frac{1}{(n-2)!}\text{ad}_{e_1}^{n-2}a_{-2,2}.
\eea
In the representation on 3D Young diagrams,
\bea
P_{n,2}&=&\frac{1}{(n-2)!}\text{ad}_{e_1}^{n-2}p_{2,2}.
\eea
\end{prp}
Similarly to the creation operators, the annihilation operators satisfy the following relations.
\begin{prp}
\bea
&&[f_1,a_{n,1}]=na_{n+1,1},\ \ \text{for }\ \ n\neq 0,
\eea
and
\bea
a_{n,2}&=&\frac{1}{(n-2)!}\text{ad}_{f_1}^{n-2}a_{2,2},
\eea
with
\bea
a_{2,2}=-[f_2,f_0]+\sigma_3 [f_1,f_0]-f_{0}^2-2\sum_{n=1}^\infty a_{-n}a_{n+2}.
\eea
\end{prp}
Acting on 3D Young diagrams, the representation of $a_{n,2}$ is denoted by $P_{n,2}^\bot$. Acting on the dual vacuum state,
\[
\langle 0|P_{2,2}^\bot=\langle 0|a_{2,2}=\langle 0|(f_0f_2-\sigma_2f_0f_1-f_1f_0),
\]
which equals the dual state of $P_{2,2}|0\rangle=(e_2e_0-\sigma_3e_1e_0-e_0e_0)|0\rangle$. We have know that\cite{3-schur}
\beaa
\langle 0|P_{2,2}^\bot P_{2,2}|0\rangle&=&\langle 0|(f_0f_2-\sigma_2f_0f_1-f_1f_0)(e_2e_0-\sigma_3e_1e_0-e_0e_0)|0\rangle\\
&=& -2(1+\sigma_2+\sigma_3^2).
\eeaa

Since
\[
f_1 P_{2,2}|0\rangle=0,
\]
we have
\beaa
&&\langle 0|(f_0f_2-\sigma_2f_0f_1-f_1f_0)\underbrace{f_1\cdots f_1}_{n+1}\underbrace{e_1\cdots e_1}_{n+1}(e_2e_0-\sigma_3e_1e_0-e_0e_0)|0\rangle\nn\\
&=&\langle 0|(f_0f_2-\sigma_2f_0f_1-f_1f_0)\underbrace{f_1\cdots f_1}_{n}f_1e_1\underbrace{e_1\cdots e_1}_{n}(e_2e_0-\sigma_3e_1e_0-e_0e_0)|0\rangle\nn\\
&=&\langle 0|(f_0f_2-\sigma_2f_0f_1-f_1f_0)\underbrace{f_1\cdots f_1}_{n}e_1f_1\underbrace{e_1\cdots e_1}_{n}(e_2e_0-\sigma_3e_1e_0-e_0e_0)|0\rangle\nn\\
&-&2(n+2)\langle 0|(f_0f_2-\sigma_2f_0f_1-f_1f_0)\underbrace{f_1\cdots f_1}_{n}\underbrace{e_1\cdots e_1}_{n}(e_2e_0-\sigma_3e_1e_0-e_0e_0)|0\rangle\nn\\
&=&-(n+4)(n+1)\langle 0|(f_0f_2-\sigma_2f_0f_1-f_1f_0)\underbrace{f_1\cdots f_1}_{n}\underbrace{e_1\cdots e_1}_{n}(e_2e_0-\sigma_3e_1e_0-e_0e_0)|0\rangle.
\eeaa
From
\[
\langle 0|a_{n,2}=\frac{(-1)^n}{(n-2)!}\langle 0|a_{2,2}\underbrace{f_1\cdots f_1}_{n-2},
\]
we have
\bea
\langle 0|a_{n,2}a_{-n,2}|0\rangle=\frac{n+1}{n-2}\langle 0|a_{n-1,2}a_{-(n-1),2}|0\rangle.
\eea
Then
\bea
\langle 0|a_{n,2}a_{-n,2}|0\rangle=-2\left(\begin{array}{cc}{n+1}\\{n-2}\end{array}\right)(1+\sigma_2+\sigma_3^2).
\eea

Since
\[
L_2=-\frac{1}{2}\left([f_2,f_0]-\sigma_3 [f_1,f_0]\right), \ L_{-2}=\frac{1}{2}\left([e_2,e_0]-\sigma_3 [e_1,e_0]\right),
\]
we have
\[
[L_2,a_{n,1}]=-n a_{n+2,1},\ \ [L_{-2}, a_{n,1}]=-na_{n-2,1}.
\]
Then
\beaa
[L_{2},\bar{L}_{-2}]&=&\frac{1}{2}[L_2, a_{-1,1}^2+2\sum_{n=1}^\infty a_{-(n+2),1}a_{n,1}]\\
&=& \frac{1}{2}(a_{1,1}a_{-1,1}+a_{-1,1}a_{1,1})+\sum_{n=1}^\infty \left((n+2)a_{-n,1}a_{n,1}-na_{-(n+2),1}a_{n+2,1}\right)\\
&=&\frac{1}{2}+4\bar{L}_0.
\eeaa
Similarly, $[\bar{L}_{2},L_{-2}]=\frac{1}{2}+4\bar{L}_0.$
Therefore,
\be\label{a22a-22}
[a_{2,2}, a_{-2,2}]=8a_{0,2}-2(1+\sigma_2+\sigma_3^2),
\ee
where
\[
a_{0,2}=\psi_2-\sum_{j=1}^\infty a_{-j}a_j.
\]
This commutation relation (\ref{a22a-22}) is the same with that in (\ref{a22a-22}). Other relations can be calculate this way. By (\ref{a22a-22}), we have
\be
\langle P_{2,2}^{n+1},P_{2,2}^{n+1}\rangle=(n+1)\left(-2(1+\sigma_2+\sigma_3^2)+16n\right)\langle P_{2,2}^{ n},P_{2,2}^{n}\rangle.
\ee
Then
\be
\langle P_{2,2}^{ n},P_{2,2}^{n}\rangle=n!\prod_{j=1}^n\left(-2(1+\sigma_2+\sigma_3^2)+16(j-1)\right).
\ee

For $n\geq j$, we also have
\be
\langle P_{n,j},P_{n,j}\rangle=\left(\begin{array}{cc}{n+j-1}\\{n-j}\end{array}\right)\langle P_{j,j},P_{j,j}\rangle.
\ee
\section{3-Jack polynomials}\label{sect5}
In this section, we want to obtain the expression of 3-Jack polynomials $\tilde{J}_\pi$ for any 3D Young diagram $\pi$. It is known that Schur functions $S_\lambda$ can be determined by
\bea
e^{\sum_{n=1}^\infty\frac{p_n}{n}z^n}&=&\sum_{n\geq 0}S_nz^n \label{pnsnschur},\\
S_{\lambda}&=&\det(S_{\lambda_i-i+j})_{1\leq i,j\leq l}\label{detsn}
\eea
for $\lambda=(\lambda_1,\lambda_2,\cdots,\lambda_2)$. The formula (\ref{detsn}) is equivalent to the Pieri formula $S_nS_\lambda=\sum_{\mu}C_{n,\lambda}^\mu S_\mu$, where the Pieri formula can be found in \cite{Mac}. For example,
 \[
 S_{n,1}=\det\left(\begin{array}{cc} S_n &S_{n+1}\\ 1 &S_{1}\end{array}\right)
 \]
from (\ref{detsn}), which is the same with the Pieri formula
\[
S_1S_n=S_{n+1}+S_{n,1}.
\]
 Here we treat 2D Young diagrams $\lambda$ as the special 3D Young diagrams which have one layer in $z$-axis direction. For 3-Jack polynomials, we need to know the formula for $\tilde{J}_\lambda$ similar to (\ref{pnsnschur}), and the formula $\tilde{J}_\lambda \tilde{J}_\pi$ similar to the Pieri formula $S_nS_\lambda$.

 In \cite{3-schur}, we have obtain that
\bea
\tilde{J}_{\twoboxy}&=&\frac{1}{(h_1-h_2)(h_1-h_3)}\left((1+h_2h_3)P_1^2+(1+h_2h_3)h_1P_{2,1}+P_{2,2}\right),\label{s(1,1)p}\\
 \tilde{J}_{\twoboxx}&=&\frac{1}{(h_2-h_1)(h_2-h_3)}\left((1+h_1h_3)P_1^2+(1+h_1h_3)h_2P_{2,1}+P_{2,2}\right),\label{s11p}\\
\tilde{J}_{\twoboxz}&=&\frac{1}{(h_3-h_1)(h_3-h_2)}\left((1+h_1h_2)P_1^2+(1+h_1h_2)h_3P_{2,1}+P_{2,2}\right).\label{s(2)p}
\eea
Note that here $P_{2,2}$ equals $\sqrt{1+\sigma_2+\sigma_3^2}P_{2,2}$ in \cite{3-schur} since in this paper we want \[\langle P_{2,2},P_{2,2}\rangle=-2(1+\sigma_2+\sigma_3^2).\]
Similarly, we let $P_{3,2}$ here equal $2\sqrt{1+\sigma_2+\sigma_3^2}P_{3,2}$, since we want $P_{3,2}=e_1P_{2,2}$, which means
 \[\langle P_{3,2},P_{3,2}\rangle=-8(1+\sigma_2+\sigma_3^2).\]

We can see that they are symmetric about three coordinate axes, which means that exchanging $\twoboxy\leftrightarrow\twoboxx$ corresponds to exchanging $h_1\leftrightarrow h_2$, others are similar. We associate $h_1$ to $y$-axis, $h_2$ to $x$-axis, $h_3$ to $z$-axis to match the results in \cite{Pro}. We want this symmetry holds for all 3-Jack polynomials.

We want that 3-Jack polynomials $\tilde{J}_\pi$ behave the same with 3D Young diagrams $\pi$ in the representation of affine Yangian of $\mathfrak{gl}(1)$. For example,
\be
\langle \tilde{J}_\pi,\tilde{J}_{\pi'}\rangle=\langle \pi,{\pi'}\rangle
\ee

In \cite{3-schur}, we show that the 3-Jack polynomials become Jack polynomials defined on 2D Young diagrams when $h_1=\sqrt{\alpha},\ h_2=-1/\sqrt{\alpha}$. In fact, if we calculate the 3-Jack polynomials for general $\psi_0$, the 3-Jack polynomials will become the symmetric functions $Y_\lambda$ when $\psi_0=-\frac{1}{h_1h_2}$, where $Y_\lambda$ are defined by us in \cite{YBalgebra,WBCW}. In this paper, we show that 3-Jack polynomials become $Y_\lambda$ under other conditions. We see that when
\be
P_{2,2}=-(1+h_1h_2)(P_1^2+h_3P_{2,1}),
\ee
the 3-Jack polynomials of two boxes become
\bea
\tilde{J}_{\twoboxy}&=&\frac{1}{h_1-h_2}\left(P_{2,1}-h_2P_1^2\right)=Y_{\begin{tikzpicture}
\draw [step=0.2](0,0) grid(.4,.2);
\end{tikzpicture}},\\
 \tilde{J}_{\twoboxx}&=&\frac{1}{h_2-h_1}\left(P_{2,1}-h_1P_1^2\right)=Y_{\begin{tikzpicture}
\draw [step=0.2](0,0) grid(.2,.4);
\end{tikzpicture}},\\
\tilde{J}_{\twoboxz}&=&0.
\eea
Generally, we take
\bea
P_{n,2}=-2(1+h_1h_2)(P_1P_{n-1,1}+h_3P_{n,1}),\ \ \text{for}\ \ n>2.
\eea
For $j=3$, we take
\bea
P_{3,3}&=& (1+h_1h_2)(2+h_1h_2)(2P_1^3+3h_3P_1P_{2,1}+h_3^2P_{3,1}),\\
P_{n,3}&=&3(1+h_1h_2)(2+h_1h_2)(2P_1^2P_{n-2,1}+3h_3P_1P_{n-1,1}+h_3^2P_{n,1}), \ \ \text{for} \ n>3.\nn
\eea
For general $j$, let
\[
(1+x)(2+x)\cdots (j-1+x)=r_0+r_1x+\cdots +r_{j-1}x^{j-1}
\]
with the coefficients $r_0=(j-1)!,r_1,\cdots, r_{n-2}, r_{j-1}=1$,
we take
\bea
P_{j,j}&=&(-1)^{j-1} \prod_{k=1}^{j-1}(k+h_1h_2)(r_0P_1^j+r_1h_3P_1^{j-2}P_{2,1}+\cdots+r_{n-1}h_3^{j-1}P_{j,1}),\label{PnjPjj}\\
P_{n,j}&=&(-1)^{j-1}j\prod_{k=1}^{j-1}(k+h_1h_2)(r_0P_1^{j-1}P_{n-j+1,1}+r_1h_3P_1^{j-2}P_{n-j+2,1}+\cdots+r_{j-1}h_3^{j-1}P_{n,1}), \nn\\
&& \text{for} \ n>j.\nn
\eea
We require the 3-Jack polynomials become the symmetric functions $Y_\lambda$ under these conditions.

For $n\geq j$, define
\be
\xi_y(P,z)=\sum_{n,j=1}^\infty \frac{ P_{n,j}}{j!\left(\begin{array}{cc}{n+j-1}\\{n-j}\end{array}\right)}\frac{d_{n,j}}{h_1^j}\prod_{k=1}^{j-1}\frac{1}{k+h_2h_3}z^n,
\ee
with
\be
d_{n,j}=\begin{cases}
 1 & \text{ if } \ n=j, \\
  j& \text{ if}\ n>j.
\end{cases}
\ee
The 3D Young diagram of $n$ boxes along $y$-axis is denoted by $\underbrace{(1,1,\cdots,1)}_n$. For example, when $n=2$, $(1,1)$ is $\twoboxy$. The 3-Jack polynomials $\tilde{J}_{\underbrace{(1,\cdots,1)}_n}$ is determined by
\be
e^{\xi_y(P,z)}=\sum_{n\geq 0}\frac{1}{\langle\tilde{J}_{\underbrace{(1,\cdots,1)}_n},\tilde{J}_{\underbrace{(1,\cdots,1)}_n}\rangle h_1^n }\tilde{J}_{\underbrace{(1,\cdots,1)}_n}(P)z^n.
\ee
Note that when $P_{n,j>1}=0$, the vertex operator above becomes that for the symmetric functions $Y_{(n)}$\cite{symmetricdeformedKP}. When $P_{n,j>1}=0$ and $h_1=\sqrt{\alpha},h_2=-1/\sqrt{\alpha}$, the vertex operator above becomes that for the 2D Jack polynomials $\tilde{J}_{(n)}$\cite{deformedKPJack}.  When $P_{n,j>1}=0$ and $h_1=-1,h_2=-1$, the vertex operator above becomes that for the Schur functions $S_{(n)}$\cite{MJD,Mac}.

We list the first few terms of $\tilde{J}_{\underbrace{(1,\cdots,1)}_n}(P)$:
\beaa
\tilde{J}_{0}&=&1,\\
\frac{1}{\langle\tilde{J}_{\onebox},\tilde{J}_{\onebox}\rangle h_1 }\tilde{J}_{\onebox}&=& \frac{1}{h_1}P_{1},\\
\frac{1}{\langle\tilde{J}_{\twoboxy},\tilde{J}_{\twoboxy}\rangle h_1^2 }\tilde{J}_{\twoboxy}&=& \frac{1}{2(1+h_2h_3)h_1^2}\left((1+h_2h_3)P_1^2+(1+h_2h_3)h_1P_{2,1}+P_{2,2}\right),\\
\frac{1}{\langle\tilde{J}_{(1,1,1)},\tilde{J}_{(1,1,1)}\rangle h_1^3 }\tilde{J}_{(1,1,1)}&=& \frac{1}{6h_{1}^{3}}P_1^3+\frac{1}{2h_1^2}P_1P_{2,1}+\frac{1}{2h_1^3(1+h_2h_3)}P_1P_{2,2}+\frac{1}{3h_1}P_{3,1}\\
&&+{\frac {1}{4h_{1}^{2} \left( h_{2}\,h_{3}+1 \right) }}P_{3,2}+{\frac {1}{6 \left( h_{2}\,h_{3}+1 \right)  \left( h_{2}\,h_{3}+2
 \right) h_{1}^{3}}}
 P_{3,3},\\
 \frac{1}{\langle\tilde{J}_{(1,1,1,1)},\tilde{J}_{(1,1,1,1)}\rangle h_1^3 }\tilde{J}_{(1,1,1,1)}&=& \frac{1}{24h_{1}^{4}}P_1^4+\frac{1}{4h_1^3}P_1^2P_{2,1}+\frac{1}{4h_1^4(1+h_2h_3)}P_1^2P_{2,2}+\frac{1}{3h_1^2}P_1P_{3,1}\\
&&+{\frac {1}{4h_{1}^{3} \left( h_{2}\,h_{3}+1 \right) }}P_1P_{3,2}+{\frac {1}{6 \left( h_{2}\,h_{3}+1 \right)  \left( h_{2}\,h_{3}+2
 \right) h_{1}^{4}}}
 P_1P_{3,3}\\
 &&+\frac{1}{4h_1}P_{4,1}+\frac{1}{10h_1^2(1+h_2h_3)}P_{4,2}+\frac{1}{12\left( h_{2}\,h_{3}+1 \right)  \left( h_{2}\,h_{3}+2
 \right) h_{1}^{3}}P_{4,3}\\
 &&+\frac {1}{24 \left( h_{2}\,h_{3}+3 \right)  \left( h_{2}\,h_{3}+2
 \right)  \left( h_{2}\,h_{3}+1 \right) {h_{1}}^{4}}P_{4,4}\\
 &&+\frac{1}{4(1+h_2h_3)h_1^3}P_{2,1}P_{2,2}+\frac{1}{8h_1^2}P_{2,1}^2+\frac{1}{8(1+h_2h_3)^2h_1^4}P_{2,2}^2.
\eeaa
Since
\[
\langle\tilde{J}_{\underbrace{(1,\cdots,1)}_{n+1}},\tilde{J}_{\underbrace{(1,\cdots,1)}_{n+1}}\rangle=\prod_{j=1}^n\frac{(j+1)(j+h_2h_3)}{(jh_1-h_2)(jh_1-h_3)},
\]
we have
\beaa
\tilde{J}_{\onebox}&=&P_1,\\
\tilde{J}_{\twoboxy}&=&\frac{1}{(h_1-h_2)(h_1-h_3)}\left((1+h_2h_3)P_1^2+(1+h_2h_3)h_1P_{2,1}+P_{2,2}\right),\\
\tilde{J}_{(1,1,1)}&=&\frac{1}{(h_1-h_2)(h_1-h_3)(2h_1-h_2)(2h_1-h_3)}\left((1+h_2h_3)(2+h_2h_3)P_1^3\right.\nn\\
&&+3h_1(1+h_2h_3)(2+h_2h_3)P_1P_{2,1}+3(2+h_2h_3)P_1P_{2,2}\nn\\
&&+2h_1^2(1+h_2h_3)(2+h_2h_3)P_{3,1}+3h_1(1+\frac{1}{2}h_2h_3)P_{3,2}\left.+P_{3,3}\right),\nn
\eeaa
which are the same with that in \cite{3-schur}.
\beaa
\tilde{J}_{(1,1,1,1)}&=&\frac{1}{(h_1-h_2)(h_1-h_3)(2h_1-h_2)(2h_1-h_3)(3h_1-h_2)(3h_1-h_3)}\\
&&\left((1+h_2h_3)(2+h_2h_3)(3+h_2h_3)P_1^4\right.+6h_1(1+h_2h_3)(2+h_2h_3)(3+h_2h_3)P_1^2P_{2,1}\\
&&+6h_1(2+h_2h_3)(3+h_2h_3)P_{2,1}P_{2,2}+6(2+h_2h_3)(3+h_2h_3)P_1^2P_{2,2}\\
&&+8(1+h_2h_3)(2+h_2h_3)(3+h_2h_3)h_1^2P_1P_{3,1}+6(2+h_2h_3)(3+h_2h_3)h_1P_1P_{3,2}\\
&&+4(3+h_2h_3)P_1P_{3,3}+6(1+h_2h_3)(2+h_2h_3)(3+h_2h_3)h_1^3P_{4,1}\\
&&+\frac{12}{5}(2+h_2h_3)(3+h_2h_3)h_1^2P_{4,2}+2(3+h_2h_3)h_1P_{4,3}+P_{4,4}\\
&&+\frac{3(2+h_2h_3)(3+h_2h_3)}{(1+h_2h_3)}P_{2,2}^2+3(1+h_2h_3)(2+h_2h_3)(3+h_2h_3)h_1^2P_{2,1}^2\left.\right),
\eeaa
which is slightly different from that in \cite{WLin} since here we choose
\beaa
P_{2,2}^2|0\rangle&=& (e_2e_0e_2e_0+2\sigma_3e_0e_1e_2e_0-2\sigma_2e_1e_0e_2e_0|0\rangle-\frac{2}{3}\sigma_3e_1e_1e_1e_0
-2e_0e_0e_2e_0\\
&&-(2+\sigma_3^2)e_0e_1e_1e_0-e_0e_2e_2e_0+\sigma_3^2e_1e_0e_1e_0
+2\sigma_3e_0e_0e_1e_0+e_0e_0e_0e_0)|0\rangle.
\eeaa

For $n\geq j$, define
\bea
\xi_x(P,z)&=&\sum_{n,j=1}^\infty \frac{ P_{n,j}}{j!\left(\begin{array}{cc}{n+j-1}\\{n-j}\end{array}\right)}\frac{d_{n,j}}{h_2^j}\prod_{k=1}^{j-1}\frac{1}{k+h_1h_3}z^n,\\
\xi_z(P,z)&=&\sum_{n,j=1}^\infty \frac{ P_{n,j}}{j!\left(\begin{array}{cc}{n+j-1}\\{n-j}\end{array}\right)}\frac{d_{n,j}}{h_3^j}\prod_{k=1}^{j-1}\frac{1}{k+h_1h_2}z^n.
\eea
The 3D Young diagram of $n$ boxes along $x$-axis and $z$-axis are denoted by ${\left(\begin{array}{ccc} 1 \\ \cdots\\1\end{array}\right)}$ and $(n)$ respectively. For example, when $n=2$, ${\left(\begin{array}{cc} 1 \\1\end{array}\right)}$ is $\twoboxx$, and $(2)$ is $\twoboxz$. The 3-Jack polynomials $\tilde{J}_{{\left(\begin{array}{ccc} 1 \\ \cdots\\1\end{array}\right)}}$ and $\tilde{J}_{(n)}$ are determined by
\be
e^{\xi_x(P,z)}=\sum_{n\geq 0}\frac{1}{\langle\tilde{J}_{{\left(\begin{array}{ccc} 1 \\ \cdots\\1\end{array}\right)}},\tilde{J}_{{\left(\begin{array}{ccc} 1 \\ \cdots\\1\end{array}\right)}}\rangle h_2^n }\tilde{J}_{{\left(\begin{array}{ccc} 1 \\ \cdots\\1\end{array}\right)}}(P)z^n
\ee
and
\be
e^{\xi_z(P,z)}=\sum_{n\geq 0}\frac{1}{\langle\tilde{J}_{(n)},\tilde{J}_{(n)}\rangle h_3^n }\tilde{J}_{(n)}(P)z^n
\ee
respectively. We can see that they are symmetric about the three coordinate axes.

For 2D Young diagrams $\lambda$, which are treated as 3D Young diagrams which have one layer in $z$-axis direction, we define 3-Jack polynomials $\tilde{J}_\lambda$ in the following. Let
\be
P_{n,1}=\frac{1}{h_1}(z_1^n+z_2^n+\cdots)=\frac{1}{h_1}\sum_{k=1}^\infty z_k^n,
\ee
then $P_{n,j}$ in (\ref{PnjPjj}) equal
\beaa
P_{j,j}&=&(-1)^{j-1}\prod_{k=1}^{j-1}(k+h_1h_2)(k+h_1h_3)\frac{1}{h_1^j}\left(\right.\sum_k z_k^j\\
&&+\frac{r_0C_j^1+r_1 h_1h_3C_{j-2}^1+r_2h_1^2h_3^2C_{j-3}^1+\cdots+r_{j-2}h_1^{j-2}h_3^{j-2}C_{1}^1}{\prod_{k=1}^{j-1}(k+h_1h_3)}\sum_{k,l}z_k^{j-1}z_l\\
&&+\frac{r_0C_j^2+r_1 h_1h_3(C_{j-2}^2+C_{j-2}^0)+r_2h_1^2h_3^2C_{j-3}^2+\cdots+r_{j-3}h_1^{j-3}h_3^{j-3}C_{2}^2}{\prod_{k=1}^{j-1}(k+h_1h_3)}\sum_{k\neq l}z_k^{j-2}z_l^2\\
&&+\frac{r_0C_j^1C_{j-1}^1+r_1 h_1h_3(C_{j-2}^1C_{j-3}^1+C_{j-2}^0)+r_2h_1^2h_3^2C_{j-3}^1C_{j-4}^1+\cdots+r_{j-3}h_1^{j-3}h_3^{j-3}C_{2}^1C_1^1}{\prod_{k=1}^{j-1}(k+h_1h_3)}\\
&&\cdot\sum_{k_1,k_2,k_3}z_{k_1}^{j-2}z_{k_2}z_{k_3}+\cdots+r_0C_{j}^1C_{j-1}^1\cdots C_{1}^1\sum_{k_1<\cdots <k_j}z_{k_1}z_{k_2}\cdots z_{k_j}\left.\right),
\eeaa
and when $n>j$,
\beaa
P_{n,j}&=&(-1)^{j-1}j\prod_{k=1}^{j-1}(k+h_1h_2)(k+h_1h_3)\frac{1}{h_1^j}\left(\right.\sum_k z_k^n+\frac{\sum_{k=0}^{j-2}r_kh_1^kh_3^kC_{j-k-1}^1}{\prod_{k=1}^{j-1}(k+h_1h_3)}\sum_{k,l}z_k^{n-1}z_l\\
&&+\frac{\sum_{k=0}^{j-3}r_kh_1^kh_3^kC_{j-k-1}^2+r_0\delta_{n-j,1}}{\prod_{k=1}^{j-1}(k+h_1h_3)}\sum_{k\neq l}z_k^{n-2}z_l^2\\
&&+\frac{\sum_{k=0}^{j-3}r_kh_1^kh_3^kC_{j-k-1}^1
C_{j-k-2}^1}{\prod_{k=1}^{j-1}(k+h_1h_3)} \sum_{k_1,k_2,k_3}z_{k_1}^{n-2}z_{k_2}z_{k_3}+\cdots \\
&&+r_0C_{j-1}^1C_{j-2}^1\cdots C_{1}^1\sum_{k_1,\cdots,k_{j+1}}z_{k_1}^{n-j+1}z_{k_2}\cdots z_{k_{j+1}}\left.\right),
\eeaa
where
\[
C_j^k=\frac{j!}{k!(j-k)!}.
\]

Define $\xi_{yx,j,j}$ and $\xi_{yx,n,j}$ by
\bea
\xi_{yx,j,j}&=&\frac{P_{j,j}}{j!h_1^j}\left(\frac{r_0C_j^1+r_1 h_1h_3C_{j-2}^1+r_2h_1^2h_3^2C_{j-3}^1+\cdots+r_{j-2}h_1^{j-2}h_3^{j-2}C_{1}^1}{\prod_{k=1}^{j-1}(k+h_1h_3)}\sum_{k,l}z_k^{j-1}z_l\right.\nn\\
&&+\frac{r_0C_j^2+r_1 h_1h_3(C_{j-2}^2+C_{j-2}^0)+r_2h_1^2h_3^2C_{j-3}^2+\cdots+r_{j-3}h_1^{j-3}h_3^{j-3}C_{2}^2}{\prod_{k=1}^{j-1}(k+h_1h_3)}\sum_{k\neq l}z_k^{j-2}z_l^2\nn\\
&&+\frac{r_0C_j^1C_{j-1}^1+r_1 h_1h_3(C_{j-2}^1C_{j-3}^1+C_{j-2}^0)+r_2h_1^2h_3^2C_{j-3}^1C_{j-4}^1+\cdots+r_{j-3}h_1^{j-3}h_3^{j-3}C_{2}^1C_1^1}{\prod_{k=1}^{j-1}(k+h_1h_3)}\nn\\
&&\left.\cdot\sum_{k_1,k_2,k_3}z_{k_1}^{j-2}z_{k_2}z_{k_3}+\cdots+r_0C_{j}^1C_{j-1}^1\cdots C_{1}^1\sum_{k_1<\cdots <k_j}z_{k_1}z_{k_2}\cdots z_{k_j}\right),
\eea
and for $n>j$,
\bea
\xi_{yx,n,j}&=&\frac{P_{n,j}}{j!C_{n+j-1}^{n-j}h_1^j}\left(\frac{\sum_{k=0}^{j-2}r_kh_1^kh_3^kC_{j-k-1}^1}{\prod_{k=1}^{j-1}(k+h_1h_3)}\sum_{k,l}z_k^{n-1}z_l\right.\nn\\
&&+\frac{\sum_{k=0}^{j-3}r_kh_1^kh_3^kC_{j-k-1}^2+r_0\delta_{n-j,1}}{\prod_{k=1}^{j-1}(k+h_1h_3)}\sum_{k\neq l}z_k^{n-2}z_l^2\nn\\
&&+\frac{\sum_{k=0}^{j-3}r_kh_1^kh_3^kC_{j-k-1}^1
C_{j-k-2}^1}{\prod_{k=1}^{j-1}(k+h_1h_3)} \sum_{k_1,k_2,k_3}z_{k_1}^{n-2}z_{k_2}z_{k_3}+\cdots \nn\\
&&+\left.r_0C_{j-1}^1C_{j-2}^1\cdots C_{1}^1\sum_{k_1,\cdots,k_{j+1}}z_{k_1}^{n-j+1}z_{k_2}\cdots z_{k_{j+1}}\right).
\eea
Define
\be
T_{yx}(P,z)=\sum_{k=1}^\infty\xi_{y}(P,z_k)+\sum_{j\geq 2}\xi_{yx,j,j}+\sum_{n>j\geq 2}\xi_{yx,n,j},
\ee
and let
\bea
e^{T_{yx}(P,z)}=\sum_{i_1,i_2,\cdots,\geq 0}Q_{i_1,i_2,\cdots}(P)z_1^{i_1}z_{2}^{i_2}\cdots,
\eea
where $z=(z_1,z_2,\cdots)$. We can see that $Q_{i_1,i_2,\cdots}$ are polynomials of $P_{n,j}$ which can be determined by this equation. We list the first few of them:
\beaa
Q_1&=&\frac{1}{h_1}P_1,\\
Q_2&=&{\frac {1}{
2 \left( h_{1}\,h_{3}+1 \right)  \left( h_{2}\,h_{3}+1 \right) h_{1}^
{2}}}\left(h_{1}^{2}h_{2}h_{3}^{2}{ P_{2,1}}+h_{1}h_2h_{
3}^{2}{ P_1}^{2}+h_{1}^{2}h_{3}{ P_{2,1}}+h_{1}h_{2}h_{3}
{ P_{2,1}}\right.\\
&&\left.+h_{1}h_{3}{{ P_1}}^{2}+h_{2}h_{3} P_1^{2}+
h_{1}h_{3}{ P_{2,2}}+h_{1}{ P_{2,1}}+ P_1^{2}+{ P_{2,2}}\right),\eeaa
\beaa
Q_{3}&=&\frac{1}{12(1+h_1h_3)(1+h_2h_3)(2+h_1h_3)(2+h_2h_3)h_1^3}\left(4h_{1}^{4}h_{2}^{2}h_{3}^{4}{ P_{3,1}}+6h_{1}^{3}h_{2}^{2
}h_{3}^{4}{ P_1}{ P_{2,1}}\right.\\
&&+2h_{1}^{2}h_{2}^{2}h_{3}^{4}{
 P_1}^{3}+12h_{1}^{4}h_{2}h_{3}^{3}{ P_{3,1}}+12h_{1}^{3
}h_{2}^{2}h_{3}^{3}{ P_{3,1}}+18h_{1}^{3}h_{2}h_{3}^{3}
 P_1P_{2,1}+18h_{1}^{2}h_{2}^{2}h_{3}^{3}{P_1}
 P_{2,1}\\
 &&+6h_{1}^{2}h_{2}h_{3}^{3} P_1^{3}+6h_{1}h_{
2}^{2}h_{3}^{3}P_1^{3}+3h_{1}^{3}h_{2}h_{3}^{3}{
P_{3,2}}+6h_{1}^{2}h_{2}h_{3}^{3}{ P_1}{ P_{2,2}}+8h_{1}^{4
}h_{3}^{2}{ P_{3,1}}\\
&&+36h_{1}^{3}h_{2}h_{3}^{2}{ P_{3,1}}+12
h_{1}^{3}h_{3}^{2}{ P_1}{ P_{2,1}}+8h_{1}^{2}h_{2}^{2}h_{
3}^{2}{ P_{3,1}}+54h_{1}^{2}h_{2}h_{3}^{2}{ P_1} P_{2,1}+
4h_{1}^{2}h_{3}^{2} P_1^{3}\\
&&+12h_{1}h_{2}^{2}h_{3}^
{2}{ P_1} P_{2,1}+18h_{1}h_{2}h_{3}^{2} P_1^{3}+4
h_{2}^{2}h_{3}^{2} P_1^{3}+6h_{1}^{3}h_{3}^{2}{ P_{3,2}
}+9h_{1}^{2}h_{2}h_{3}^{2}{ P_{3,2}}\\
&&+12h_{1}^{2}h_{3}^{2}
{ P_1}{ P_{2,2}}+18h_{1}h_{2}h_{3}^{2}{ P_1}{P_{2,2}}+
24h_{1}^{3}h_{3}{ P_{3,1}}+24h_{1}^{2}h_{2}h_{3}{ P_{3,1}}
+2h_{1}^{2}h_{3}^{2}{ P_{3,3}}\\
&&+36h_{1}^{2}h_{3} P_1 P_{2,1}+36h_{1}h_{2}h_{3}{ P_1}{ P_{2,1}}+12h_{1}h_{3}
 P_1^{3}+12h_{2}h_{3} P_1^{3}+18h_{1}^{2}h_{3}
 P_{3,2}\\
 &&+6h_{1}h_{2}h_{3}{ P_{3,2}}+36h_{1}h_{3}{
P_1}{ P_{2,2}}+12h_{2}h_{3}{ P_1}{ P_{2,2}}+16h_{1}^{2}{
 P_{3,1}}+6h_{1}h_{3}{ P_{3,3}}\\
 &&+24h_{1}{ P_1}{ P_{2,1}}+8
 P_1^{3}+12h_{1}{ P_{3,2}}+\left.24{ P_1}{ P_{2,2}}+4{
P_{3,3}}
\right),
\eeaa
and
\beaa
Q_{1,1}&=&{\frac {1}{
 \left( h_{1}\,h_{3}+1 \right)  \left( h_{2}\,h_{3}+1 \right) h_{1}^
{2}}}\left(\left( h_{1}\,h_{3}+1 \right)  \left( h_{2}\,h_{3}+1 \right)P_1^2+P_{2,2}\right),\\
Q_{2,1}&=&\frac{1}{4(1+h_1h_3)(1+h_2h_3)(2+h_1h_3)h_1^3}\left(2h_1^2h_2^2h_3^3P_1P_{2,1}+2h_1h_2^2h_3^3P_1^3\right.\\
&&+6h_1^2h_2h_3^2P_1P_{2,1}+2h_1h_2^2h_3^2P_1P_{2,1}+6h_1h_2h_3^2P_1^3+2h_2^2h_3^2P_1^3+2h_1h_2h_3^2P_1P_{2,2}\\
&&+4h_1^2h_3P_1P_{2,1}+6h_1h_2h_3P_1P_{2,1}+4h_1h_3P_1^3+6h_2h_3P_1^3+h_1h_2h_3P_{3,2}+4h_1h_3P_1P_{2,2}\\
&&\left.+6h_2h_3P_1P_{2,2}+4h_1P_1P_{2,1}+4P_1^3+2h_1P_{3,2}+12P_1P_{2,2}+2P_{3,3}\right).
\eeaa
Then we can calculate the 3-Jack polynomials $\tilde{J}_\lambda$ from $Q_{i_1,i_2,\cdots}$. When $(i_1,i_2,i_3,\cdots)=(n,0,0,\cdots)$, we have
\be\label{QnJn}
Q_n(P)=\frac{1}{\langle \tilde{J}_{(n)},\tilde{J}_{(n)}\rangle h_1^n}\tilde{J}_{(n)}(P).
\ee
When
$(i_1,i_2,i_3,\cdots)=(n-1,1,0,\cdots)$, we have
\bea
Q_{n-1,1}(P)&=&\frac{1}{\langle \tilde{J}_{(n)},\tilde{J}_{(n)}\rangle h_1^n}\tilde{J}_{(n)}(P)\frac{-nh_2}{(n-1)h_1-h_2}\nn\\
&+&\frac{1}{\langle \tilde{J}_{(n-1,1)},\tilde{J}_{(n-1,1)}\rangle h_1^{n-1}}\tilde{J}_{(n-1,1)}(P)\frac{2}{h_1-h_2},\label{Qn-11}
\eea
where $(n-1,1)$ is the 2D Young diagram from $(1,1)$ by adding $n-2$ box in the first row. For example,
\beaa
Q_{2}&=& \frac{1}{\langle \tilde{J}_{(2)},\tilde{J}_{(2)}\rangle h_1^2}\tilde{J}_{(2)}(P),
\eeaa
which means
\[
\tilde{J}_{(2)}(P)=\frac{1}{(h_1-h_2)(h_1-h_3)}\left((1+h_2h_3)P_1^2+(1+h_2h_3)h_1P_{2,1}+P_{2,2}\right),
\]
which is the same with (\ref{s(1,1)p}).
\beaa
Q_{1,1}=\frac{1}{\langle \tilde{J}_{(2)},\tilde{J}_{(2)}\rangle h_1^2}\tilde{J}_{(2)}(P)\frac{-2h_2}{h_1-h_2}+\frac{1}{\langle \tilde{J}_{(1,1)},\tilde{J}_{(1,1)}\rangle h_1}\tilde{J}_{(1,1)}(P)\frac{2}{h_1-h_2},
\eeaa
which means
\[
\tilde{J}_{1,1}=\frac{1}{(h_2-h_1)(h_2-h_3)}\left((1+h_1h_3)P_1^2+(1+h_1h_3)h_2P_{2,1}+P_{2,2}\right),
\]
which is the same with (\ref{s11p}).

In symmetric functions $Y_\lambda(P)$, let $p_n=\frac{1}{h_1}\sum_k z_k^n$, we see that $Y_\lambda(P)=Y_\lambda(z)$ are symmetric about $z_1,z_2,\cdots$. As in \cite{Mac}, we regard 2D Young diagrams arranged in the reverse lexicographical order $\succ$, so that $(n)$ comes first and $1^n$ comes last. We arrange the terms in $Y_{\lambda}(z)$ the same as the order of 2D Young diagrams, so that $z_i^n$ comes first and $z_{i_1}z_{i_2}\cdots z_{i_n}$ comes last. We use the notation $c_{i_1,i_2,\cdots}^{Y_{\lambda}(z)}$ to denote the coefficient of $z_1^{i_1}z_2^{i_2}\cdots$ in $Y_{\lambda}(z)$. It can be checked that the formulas (\ref{QnJn}) and (\ref{Qn-11}) can be written as
\bea
Q_n(P)&=&\frac{1}{\langle \tilde{J}_{(n)},\tilde{J}_{(n)}\rangle }\tilde{J}_{(n)}(P)c_{n}^{Y_{(n)}(z)},\\
Q_{n-1,1}(P)&=&\frac{1}{\langle \tilde{J}_{(n)},\tilde{J}_{(n)}\rangle}\tilde{J}_{(n)}(P)c_{(n-1,1)}^{Y_{(n)}(z)}\nn\\
&&+\frac{1}{\langle \tilde{J}_{(n-1,1)},\tilde{J}_{(n-1,1)}\rangle }\tilde{J}_{(n-1,1)}(P)c_{(n-1,1)}^{Y_{(n-1,1)}(z)}.
\eea
Actually, this formula holds generally, that is, for Young diagram $\lambda$, we have
\be\label{QlambdaJlambda}
Q_\lambda(P)=\sum_{\mu\succ\lambda}\frac{1}{\langle \tilde{J}_{\mu},\tilde{J}_{\mu}\rangle }\tilde{J}_{\mu}(P)c_{(\lambda_1,\lambda_2,\cdots)}^{Y_{\mu}(z)}.
\ee
For any 2D Young diagrams $\mu$, which are treated as the 3D Young diagrams having one layer in $z$-axis direction, the 3-Jack polynomials $\tilde{J}_\mu$ can be obtained from the formula (\ref{QlambdaJlambda}).

Note that we can similarly define $T_{xy}(P,z),\ T_{xz}(P,z),\ T_{zx}(P,z),\ T_{yz}(P,z),\ T_{zy}(P,z)$. From them, the 3-Jack polynomials of 3D Young diagrams having one layer in $x$-axis direction or $y$-axis direction can be obtained. In fact, the 3-Jack polynomials of 3D Young diagrams having one layer in $x$-axis direction or $y$-axis direction can also be obtained from the 3-Jack polynomials of 3D Young diagrams having one layer in $z$-axis direction by the symmetry of 3-Jack polynomials about three coordinate axes.

To get the expressions of 3-Jack polynomials $\tilde{J}_\pi$ for all 3D Young diagrams $\pi$, we need the formula $\tilde{J}_\lambda\tilde{J}_\pi$.
Define
\be
\tilde{J}_\lambda\tilde{J}_\pi=\hat{\tilde{J}}_\lambda\cdot\tilde{J}_\pi,
\ee
where $\hat{\tilde{J}}_\lambda$ are the functions of operators $a_{-n,j}$ with $n>0$, and the actions of $\hat{\tilde{J}}_\lambda$ on $\tilde{J}_\pi$ are the same with that of affine Yangian of $\mathfrak{gl}(1)$ on 3D Young diagrams. For example,
since \[
e_0|\Box\rangle=|\twoboxy\rangle+|\twoboxx\rangle+|\twoboxz\rangle,
\]
we have
\be
\tilde{J}_\Box \tilde{J}_\Box=e_0\tilde{J}_\Box=\tilde{J}_{\twoboxy}+\tilde{J}_{\twoboxx}+\tilde{J}_{\twoboxz},
\ee
then $\tilde{J}_{\twoboxz}$ is obtained, which is the same with (\ref{s(2)p}).
From
\be
\tilde{J}_\Box \tilde{J}_{\twoboxy}=\tilde{J}_{\twoboxy}\tilde{J}_\Box=\tilde{J}_{(1,1,1)}+\tilde{J}_{\left(\begin{array}{cc} 1&1 \\ 1&\end{array}\right)_{h_1,h_2}}+\tilde{J}_{(2,1)_{h_1,h_3}},
\ee
$\tilde{J}_{(2,1)_{h_1,h_3}}$ is obtained.
Other 3-Jack polynomials of 3D Young diagrams which have more than one layer in $z$-axis direction can be obtained this way.
\section{Concluding remarks}\label{sect6}
In this paper, all results are obtained by requiring $\psi_0=1$. If interested, one can calculate the results for general $\psi_0$, which should be similar to that in this paper. For example,
\[
1+\sigma_2+\sigma_3^2=(1+h_1h_2)(1+h_1h_3)(1+h_2h_3)
\]
in this paper should be
\[
1+\psi_0\sigma_2+\psi_0^3\sigma_3^2=(1+\psi_0h_1h_2)(1+\psi_0h_1h_3)(1+\psi_0h_2h_3).
\]
for general $\psi_0$.
\[
\langle\tilde{J}_{\underbrace{(1,\cdots,1)}_{n+1}},\tilde{J}_{\underbrace{(1,\cdots,1)}_{n+1}}\rangle=\prod_{j=1}^n\frac{(j+1)(j+h_2h_3)}{(jh_1-h_2)(jh_1-h_3)}
\]
in this paper should be
\[
\langle\tilde{J}_{\underbrace{(1,\cdots,1)}_{n+1}},\tilde{J}_{\underbrace{(1,\cdots,1)}_{n+1}}\rangle=\psi_0\prod_{j=1}^n\frac{(j+1)(j+h_2h_3\psi_0)}{(jh_1-h_2)(jh_1-h_3)}.
\]
This holds since there is the scaling symmetries in the affine Yangian of $\mathfrak{gl}(1)$\cite{Pro}. The scaling symmetries say that changing the value of $\psi_0$ is equivalent to rescaling parameters $h_j, j=1,2,3$. Next, we will consider the slice of 3-Jack polynomials similar to that the slices of 3D Young diagrams are 2D Young diagrams.
\section*{Data availability statement}
The data that support the findings of this study are available from the corresponding author upon reasonable request.

\section*{Declaration of interest statement}
The authors declare that we have no known competing financial interests or personal relationships that could have appeared to influence the work reported in this paper.

\section*{Acknowledgements}
This research is supported by the National Natural Science Foundation
of China under Grant No. 12101184 and No. 11871350, and supported by Key Scientific Research Project in Colleges and Universities of Henan Province No. 22B110003.

\end{document}